\documentclass[a4paper,12pt,final, reqno]{amsart}
\usepackage{times, a4wide,mathrsfs,amssymb,dsfont, enumerate, xypic}
\usepackage{hyperref}
\usepackage[usenames,dvipsnames]{xcolor}
\hypersetup{colorlinks=true,citecolor=NavyBlue,linkcolor=Brown,urlcolor=Orange}

\newcommand{\C}{\mathbb{C}}
\newcommand{\Z}{\mathbb{Z}}

\newcommand{\QQ}{\mathbb{Q}}

\newcommand{\PP}{\mathbb{P}}
\newcommand{\LL}{\mathcal{L}}

\newcommand{\OO}{\mathcal O}

\newcommand{\XX}{\mathcal X}

\newcommand{\FF}{\mathcal F}
\newcommand{\TT}{\mathcal T}
\newcommand{\CH}{\operatorname{CH}}

\newcommand{\YY}{\mathcal Y}
\newcommand{\PPP}{\mathcal P}
\newcommand{\ZZZ}{\mathcal Z}

\newcommand{\pic}{\hbox{Pic}}

\newcommand{\ima}{\operatorname{Im}}
\newcommand{\rom}{\romannumeral}

\newcommand{\ide}{\mathrm{id}}
\newcommand{\aut}{\hbox{Aut}}
\newcommand{\bir}{\hbox{Bir}}

\newif\ifHideFoot
\HideFoottrue  
\HideFootfalse  

\ifHideFoot

\newcommand{\Robert}[1]{}
\newcommand{\Charles}[1]{}

\else 

\newcommand{\marg}[1]{\normalsize{{
			\color{red}\footnote{{\color{blue}#1}}}{\marginpar[\vskip
			-.25cm{\color{red}\hfill$\Rightarrow$\tiny\thefootnote}]{\vskip
				-.2cm{\color{red}$\Leftarrow$\tiny\thefootnote}}}}}
\newcommand{\Robert}[1]{\marg{(Robert) #1}}
\newcommand{\Charles}[1]{\marg{(Charles) #1}}

\fi

\newtheorem{theorem}{Theorem}[section]
\newtheorem{claim}[theorem]{Claim}
\newtheorem{lemma}[theorem]{Lemma}
\newtheorem{corollary}[theorem]{Corollary}
\newtheorem{proposition}[theorem]{Proposition}

\newtheorem{thm2}{Theorem}
\newtheorem{conj2}{Conjecture}

\theoremstyle{definition}

\newtheorem{convention}{Conventions}

\newtheorem{remark}[theorem]{Remark}
\newtheorem{definition}[theorem]{Definition}

\newtheorem{nonumberingt}{Acknowledgements}

\begin{document}

	\author[Robert Laterveer]
	{Robert Laterveer}
	\address{Institut de Recherche Math\'ematique Avanc\'ee,
		CNRS -- Universit\'e 
		de Strasbourg,\
		7 Rue Ren\'e Des\-car\-tes, 67084 Strasbourg CEDEX,
		FRANCE.}
	\email{robert.laterveer@math.unistra.fr}

	\author[Charles Vial]
	{Charles Vial}
	\address{Universit\"at Bielefeld, Germany}
	\email{vial@math.uni-bielefeld.de}


	\title{Zero-cycles on double EPW sextics}	
	
	\begin{abstract} The Chow rings of hyperK\"ahler varieties are conjectured to
		have a particularly rich structure. In this paper, we focus on the locally
		complete family of double EPW sextics and establish some properties of their
		Chow rings.  
		First we prove a Beauville--Voisin type theorem for zero-cycles on double EPW
		sextics\,; precisely, we show that the codimension-4 part of the subring of the
		Chow ring of a double EPW sextic generated by divisors, the Chern classes and
		codimension-2 cycles invariant under the	anti-symplectic covering involution has
		rank one. Second, for double EPW sextics  birational to the Hilbert square of a
		K3 surface, we show that
		the action of the anti-symplectic involution on the Chow group of zero-cycles
		commutes with the Fourier decomposition of Shen--Vial. 
	\end{abstract}

	

	\maketitle
	
	\section*{Introduction}
	 
	Since the seminal work of Beauville and Voisin on the Chow ring of K3 surfaces
	\cite{BV}, it has been observed that the Chow rings (and more generally the Chow
	motives, considered as algebra objects) of hyperK\"ahler varieties possess a surprisingly rich structure, which
	seems to parallel that of abelian varieties. Our aim is to study aspects
	of the Chow ring, which conjecturally should hold for all hyperK\"ahler
	varieties, in the special case of double EPW sextics.
	Discovered by O'Grady \cite{OG}, double EPW sextics form a 20-dimensional
	locally complete family of hyperK\"ahler fourfolds, deformation equivalent to
	the Hilbert square of a K3 surface.
	\medskip

	For a scheme $X$ of finite type over a field, we denote
	$\CH^i(X)$
	the Chow group of codimension-$i$ cycle classes with rational coefficients
	(\emph{i.e.} the group of codimension-$i$ algebraic cycles on
	$X$ with $\QQ$-coefficients, modulo rational equivalence).

	\subsection{The Beauville--Voisin conjecture}
	
	\begin{conj2}[Beauville--Voisin]
		Let $X$ be a hyperK\"ahler variety.  Consider the $\QQ$-subalgebra
		\[ \operatorname{R}^\ast(X):= \langle \CH^1(X),  c_j(X)\rangle\ \ \ \subset\
		\CH^\ast(X)\
		\]
		generated by divisors and
		Chern classes.
		Then the restriction of the cycle class map $\operatorname{R}^i(X)\to
		H^{2i}(X,\QQ)$ is injective
		for all $i$.
	\end{conj2}
	
	The conjecture was first proven (without being stated as such) in the case of K3
	surfaces in the seminal work of Beauville and Voisin \cite{BV}.
	The conjecture was then formulated by Beauville \cite{Beau3} without the Chern
	classes as an explicit workable consequence of a deeper conjecture stipulating
	the splitting of the conjectural Bloch--Beilinson filtration on the Chow rings
	of hyperK\"ahler varieties. It was then stated in this form by Voisin \cite{V17}
	who established it in the case of Hilbert schemes of points on K3 surfaces of
	low dimension and in the case of Fano varieties of lines on smooth cubic
	fourfolds. The Beauville--Voisin conjecture has by now been established in many
	cases, including the generic double EPW sextic \cite{Fe}, generalized Kummer
	varieties \cite{FuKummer} and, without the Chern classes, hyperK\"ahler
	varieties carrying a rational Lagrangian fibration \cite{Riess} as well as Hilbert schemes of points on K3 surfaces \cite{MN}. 
	There are however, so far, only two examples of locally complete families of
	hyperK\"ahler varieties all of whose members are known to satisfy the
	Beauville--Voisin conjecture, namely K3 surfaces \cite{BV} and Fano varieties of
	lines on cubic fourfolds \cite{V17}. The following result establishes in
	particular the Beauville--Voisin conjecture for zero-cycles for double EPW
	sextics, which form a locally complete family of hyperK\"ahler varieties.
	
	\begin{thm2}\label{thm:ferretti}
		Let $X$ be a double EPW sextic, and let $\iota$ be its anti-symplectic
		involution. Consider the $\QQ$-subalgebra
		\[ \operatorname{R}^\ast(X):= \langle \CH^1(X), \CH^2(X)^+, c_j(X)\rangle\ \ \
		\subset\ \CH^\ast(X)\
		\]
		generated by divisors, $\iota$-invariant codimension-2 cycles and
		Chern classes.
		The restriction of the cycle class map $\operatorname{R}^i(X)\to
		H^{2i}(X,\QQ)$ is injective for
		$i = 4$.
	\end{thm2}
	
	This builds on and extends the main result of Ferretti \cite{Fe}.
	The reason for including $ \CH^2(X)^+$ is motivated by Beauville's splitting
	property conjecture for the Bloch--Beilinson filtration \cite{Beau3}, which in fact suggests that all of $\operatorname{R}^*(X)$ should inject in cohomology.
	Our new input consists in extending a result of Voisin concerning zero-cycles on
	generic Calabi--Yau hypersurfaces to the case of Calabi--Yau hypersurfaces with
	quotient singularities (the EPW sextics are such Calabi--Yau hypersurfaces)\,;
	this is Theorem~\ref{cy}. Unfortunately, we failed to establish the
	Beauville--Voisin conjecture for codimension-3 cycles\,; we identify in
	\eqref{eq:missing} the missing relations. \medskip

	\subsection{Anti-symplectic involutions and zero-cycles}
	In the same way that the action of homomorphisms of abelian varieties preserves
	the Beauville decomposition \cite{BeauAb} of the Chow groups (such a
	decomposition provides a splitting of the conjectural Bloch--Beilinson
	filtration), it is conceivable to expect that morphisms (or even rational maps)
	between hyperK\"ahler varieties preserve the conjectural splitting of the
	conjectural Bloch--Beilinson filtration. 
	Candidates for such a splitting were constructed for Hilbert schemes of K3
	surfaces \cite{SV, V6}, generalized Kummer varieties \cite{FTV} and Fano
	varieties of lines on cubic fourfolds \cite{SV}.
	This expectation was verified for
	the action of Voisin's rational self-map on the Fano variety of lines on a
	cubic fourfold in \cite[Proposition~21.14]{SV}, and for the action of
	finite-order symplectic automorphisms on zero-cycles of generalized Kummer
	varieties \cite[Theorem~5]{Via}. We also note that any rational map between K3
	surfaces is compatible with the splitting of the Bloch--Beilinson filtration
	given by $\CH^2(S) = \QQ[o] \oplus \CH^2_{hom}(S)$, where $o$ denotes the
	Beauville--Voisin \cite{BV} zero-cycle on $S$\,; indeed $o$ is the class of any
	point lying on a rational curve of $S$ and hence it is sent to the class of a
	point lying on a rational curve by the action of any rational map.
	We provide more evidence for this expectation by determining the action of the
	anti-symplectic involution attached to a double EPW sextic birational to the
	Hilbert square of a K3 surface, and also by determining the action of a
	birational automorphism of the Hilbert square of a very general K3 surface.
	
	Precisely,
	we show in Theorem~\ref{main} that $\iota^*$ commutes with the Fourier
	decomposition of $\CH^4(S^{[2]})$ constructed in \cite{SV}, which provides an
	explicit candidate for the splitting of the Bloch--Beilinson conjecture as
	conjectured by Beauville \cite{Beau3}. 
	As a consequence, we describe explicitly the action of $\iota$ on the Chow
	group
	of zero-cycles in case
	$X$ is birational to a Hilbert square $S^{[2]}$ with $S$ a K3 surface (which
	happens on a dense, countable union of divisors in the moduli space of double
	EPW sextics)\,:
	
	\begin{thm2}\label{main0}
		Let $X$ be a smooth double EPW sextic, and assume that $X$ is birational to a
		Hilbert square $S^{[2]}$ with $S$ a K3 surface. Let $\iota\in\aut(X)$ be the
		anti-symplectic involution coming from the double cover $f:X\to Y$, where
		$Y\subset\PP^5$ is an EPW sextic. Then the action of $\iota$ on
		$\CH^4(S^{[2]})$
		is given by 
		$$\iota^*[x,y] = [x,y] - 2[x,o] - 2[y,o] + 4[o,o],$$
		where $o\in S$ denotes any point lying on a rational curve in $S$ and where $[x,y]$ 
		denotes the class in $\CH^4(S^{[2]})$ of any point in $S^{[2]}$ with support 
		$x+y \in S^{(2)} = S^2/\mathfrak{S}_2$ (see \S \ref{secmck}).
	\end{thm2}

	Thanks to work of Boissi\`ere \emph{et alii} \cite{BC} and Debarre--Macr\`i \cite{DM},
	Theorem~\ref{main0} implies (and in fact by Remark~\ref{R:eq} is equivalent to)
	the following statement\,:
	
	\begin{thm2}[Corollary~\ref{cor0}] 
		Let $X$ be a Hilbert scheme
		$X=S^{[2]}$ where $S$ is a K3 surface with $\pic(S)=\Z$.
		Let $\iota\in\bir(X)$
		be a non-trivial birational automorphism. Then ($\iota$ is a non-symplectic
		birational involution, and) $\iota$ acts on $\CH^4(X)$ as in
		Theorem~\ref{main0}.
	\end{thm2}
	
	In \S \ref{S:applications}, we provide two applications to Theorem~\ref{main0}\,:
	in Corollary ~\ref{cor} we extend Theorem~\ref{thm:ferretti} to codimension-3
	cycles when $X$ is birational to the Hilbert square of a K3 surface, while in
	Corollary~\ref{cor2} we show that  the canonical zero-cycle can be characterized
	as the class of any point lying on a uniruled divisor whose class is
	$\iota$-invariant. \medskip
	
	There are three other explicit families of hyperK\"ahler varieties
	such that all members have an anti-symplectic involution\,: the double EPW
	quartics of
	\cite{IKKR}, the double EPW cubes of \cite{IKKR0} and the eightfolds of
	\cite{LLSS}. It would be interesting to try the argument of the present note
	for those hyperK\"ahlers that are in addition birational to a Hilbert scheme
	of a K3 surface.
	
	\begin{convention} 
		In this note, the word {\sl variety\/} will refer to an integral scheme of finite type over $\C$. 
		By \emph{quotient variety}, we will mean a finite quotient of a smooth variety. 
		For a variety $X$,
		$\CH_j(X)$ will denote the Chow group of $j$-dimensional algebraic cycles on
		$X$ 
		with $\QQ$-coefficients.
		For $X$ smooth of dimension $n$ the notations $\CH_j(X)$ and $\CH^{n-j}(X)$ will
		be
		used interchangeably. 
		We will write $H^j(X)$ to indicate singular cohomology $H^j(X(\C),\QQ)$.
		The notation
		$\CH^j_{hom}(X)$ will be used to indicate the subgroups of 
		homologically trivial cycles.
		We denote $V=V^+\oplus V^-$ the eigenspace decomposition of an involution
		acting on a vector space~$V$. 
	\end{convention}

	\section{Zero-cycles on Calabi--Yau hypersurfaces with quotient singularities}
	
	Our approach to proving Theorem~\ref{thm:ferretti}  will involve determining
	the
	subgroup of $\CH^4(X)$ that is generated by the intersection of
	$\iota$-invariant
	cycles on $X$ of positive codimension. Such cycles are pull-backs of the
	intersection of cycles of positive codimension on the EPW sextic, which is a
	Calabi--Yau hypersurface that is a quotient variety. We observe that Voisin's
	\cite[Theorem 3.4]{V13} can be generalized to the case of Calabi--Yau 
	hypersurfaces that are quotient varieties. First we recall that intersection theory on smooth varieties extends to quotient varieties if one is ready to work with rational coefficients\,:
	
	\begin{lemma}\label{quotient} Let $M$ be a {\em quotient variety\/}, \emph{i.e.}
		$M=M^\prime/G$ where $M^\prime$ is a smooth quasi-projective variety and
		$G\subset\aut(M^\prime)$ is a finite group.
		Then $\CH^\ast(M):=\oplus_i \CH^i(M):=\oplus_i\CH_{\dim M-i}(M)$ is a commutative
		graded ring, with the usual functorial properties.
	\end{lemma}
	
	\begin{proof} According to \cite[Example 17.4.10]{F}, the natural map
		\[ \CH^i(M)\ \to\ \CH_{\dim M-i}(M) \]
		from operational Chow cohomology (with $\QQ$-coefficients) to the usual Chow
		groups (with $\QQ$-coefficients) is an isomorphism. The lemma follows from the
		good formal properties of operational Chow cohomology.
	\end{proof}

	\begin{theorem}\label{cy} 
		Let $Y\subset\PP^{n+1}(\C)$ be a hypersurface
		of degree $n+2$, and assume $Y$ is a quotient variety. Then the image
		of
		the intersection product map
		\[ \CH^i(Y)\otimes \CH^{n-i}(Y)\ \to\ \CH^n(Y),\ \ \ 0<i<n,\]
		is one-dimensional.
	\end{theorem}

	\begin{proof}[Proof of Theorem~\ref{cy}]
	In the case $Y$ is a {\em
			general\/} hypersurface, this is due to Voisin \cite[\S 3]{V13} (this was extended to general
		Calabi--Yau complete intersections by L. Fu \cite{LFu}). The genericity
		assumption is only made in order to ensure that $Y$ is smooth and the Fano variety $F(Y)$ of
		lines in $Y$ is of the expected dimension (which is $n-3$). Let us
		check that
		Voisin's argument extends to all Calabi--Yau hypersurfaces that are
		quotient varieties.

First we introduce some notation. If $\XX \to B$ is a complex morphism to a smooth complex variety $B$, we denote $Z_b$ the fiber over $b\in B(\C)$ of the subscheme $\ZZZ \subseteq \XX$, while we denote $\Gamma\vert_b$ the Gysin fiber~\cite{F} in $\CH_{\ast}(X_b)$ of the cycle class $\Gamma \in \CH_{\ast}(\XX)$.
		
		Let now $B = \PP H^0(\PP_\C^{n+1},\mathcal O(n+2))$ be the space parameterizing  degree $n+2$ hypersurfaces in
		$\PP^{n+1}_\C$ and let $\YY\to B$ be the corresponding universal family, \emph{i.e.} $\YY := \{(f,x) : f(x) = 0\} \subseteq B\times \PP_\C^{n+1}$.
	Let $B^\circ \subset B$ be the non-empty
	open subset parameterizing  smooth hypersurfaces the Fano varieties of which have
	 dimension $n-3$. Voisin's argument in the proof of  \cite[Theorem~3.1]{V13}
	 provides relative cycle classes $o\in \CH^n(\YY_{B^\circ})$ and $R, \Gamma \in \CH^{2n}(\YY\times_{B^\circ}\YY\times_{B^\circ} \YY)$ such that 
		\begin{equation}\label{decomp} 
		\delta_{Y_b} = p_{12}^*\Delta_{Y_b} \cdot p_3^*o\vert_b + p_{13}^*\Delta_{Y_b} \cdot p_2^*o\vert_b + p_{23}^*\Delta_{Y_b} \cdot p_1^*o\vert_b +  R\vert_b+\Gamma\vert_b\ \ \ \hbox{in}\ \CH^{2n}(Y_b\times Y_b\times
		Y_b)
		\end{equation}
		for all $b\in B^\circ$ with the following properties\,: $o\vert_b = \frac{1}{n+2}\, h^n$ with $h$ the hyperplane class on $Y_b$, $\Delta_{Y_b}$ is the diagonal class in $\CH^n(Y_b\times Y_b)$, 	$\delta_{Y_b} := p_{12}^*\Delta_{Y_b} \cdot p_{23}^*\Delta_{Y_b}$ is the so-called small diagonal, the restriction $R\vert_b$ of $R$ at $b\in B^\circ$ is some cycle in the
		image of the restriction map 
		\[ \CH^{2n}\bigl(\PP^{n+1}\times \PP^{n+1} \times \PP^{n+1}\bigr)\to
		\CH^{2n}\bigl(Y_b\times Y_b\times
		Y_b\bigr),\]
		and $\Gamma$ is a multiple of the cycle class attached to
	\[  \ima\bigl(  \LL\times_{B^\circ} \LL\times_{B^\circ} \LL\ \to\
			\YY\times_{B^\circ} \YY\times_{B^\circ} \YY\bigr)\ \ \ \subset\    \YY\times_{B^\circ}
			\YY\times_{B^\circ} \YY\ ,\]
			where $\LL \rightarrow \YY$ is the ``relative universal line''
		over $B$. Here, $p_i : Y_b\times Y_b\times Y_b \to Y_b$ is the projection on the $i$-th factor and $p_{ij} : Y_b\times Y_b\times Y_b \to Y_b\times Y_b$ is the projection on the product of the $i$-th and $j$-th factors.
		This decomposition implies Theorem~\ref{cy} for $Y_b$ with $b\in B^\circ$ (which is \cite[Theorem~3.4]{V13}). Indeed,
		for any $\beta\in    \CH^i(Y_b)$ and any $\gamma\in    \CH^{n-i}(Y_b)$ with $0<i<n$ one
		has
		\[ \beta\cdot \gamma = (\delta_{Y_b})_\ast (\beta\times \gamma) = \deg(\beta\cdot \gamma) o\vert_b +
		(R\vert_b+\Gamma\vert_b)_\ast(\beta\times \gamma)\ \ \ \hbox{in}\ \CH^n(Y_b),\]
		where the small diagonal $\delta_{Y_b}$ is considered as a correspondence from $Y_b\times Y_b$ to $Y_b$,
		and one can check that the right-hand side is proportional to $h^n\in
		\CH^n(Y_b)$.
		
		Let now
		$\bar{R},\bar{\Gamma} \in \CH_{\ast}(\YY\times_{B}\YY\times_{B}\YY)$ be relative cycles over $B$
		restricting over $B^\circ$ to 
		$R$ and~$\Gamma$, respectively. (Here we use the $\CH_\ast(-)$ notation, since
		$\YY\to B$ is not a smooth
		morphism and so the fiber product $\YY\times_{B}\YY\times_{B}\YY$ may not be
		smooth and contain components of various dimensions.) A standard Hilbert scheme argument \cite[Proposition
		2.4]{V10}
		implies that the decomposition~\eqref{decomp} extends
		to
		the whole parameter space $B$, in the sense that 
		\begin{equation*}
		\delta_{Y_b}= p_{12}^*\Delta_{Y_b} \cdot p_3^*o\vert_b + p_{13}^*\Delta_{Y_b} \cdot p_2^*o\vert_b + p_{23}^*\Delta_{Y_b} \cdot p_1^*o\vert_b +  \bar{R}\vert_b+\bar{\Gamma}\vert_b\ \ \ \hbox{in}\
		\CH_{n}(Y_b\times Y_b\times Y_b)
		\end{equation*}
	for\ any $b\in B$.
		
		Since the formalism of correspondences with $\QQ$-coefficients goes through
		unchanged for quotient varieties, the equality $\beta\cdot
		\gamma=(\delta_{Y_b})_\ast(\beta\times \gamma)$ is still valid for quotient
		varieties. Hence, to prove the theorem
		we just need to understand the action of the correspondences $ \bar{R}\vert_b$
		and $\bar{\Gamma}\vert_b$ for $b$ parameterizing a hypersurface that is a quotient variety. The first is easy\,: the action of $
		\bar{R}\vert_b$ still
		 factors over $\CH^n(\PP^{n+1})$ and so  $
		(\bar{R}\vert_b)_\ast(\beta\times \gamma)$ is proportional to $h^n$.
		As for the second, we can consider the locus swept out by lines
		\[ \ZZZ := \ima\bigl(  \LL\times_{B} \LL\times_{B} \LL\ \to\
		\YY\times_{B} \YY\times_{B} \YY\bigr)\ \ \ \subset\    \YY\times_{B}
		\YY\times_{B} \YY\ .\]
		The natural morphism $\LL\to B$ has fibers of dimension $n-3\geq 0$ over $B^\circ$, but
		the fiber dimension may jump outside of $B^\circ$. Because of
		upper-semicontinuity,
		any fiber $L_b$ has dimension $\ge n-3$. By construction of the
		refined Gysin homomorphism \cite[\S 6]{F}, we have that
		the zero-cycle $(\bar \Gamma\vert_b)_*(\beta\times\gamma) $ is supported on
	 the image of $L_b$ under the restriction over $b$ of the natural morphism $\LL \to \YY$.
		It follows that $(\bar \Gamma\vert_b)_*(\beta\times\gamma) $ 
	is supported on a finite union of lines contained in $Y_b$. We can conclude since
		any $0$-cycle on a line is proportional to $h^n$ in $Y_b$.      
	\end{proof}

	\begin{remark}\label{rmk:h}
		The image
		of
		the intersection product map
		\[ \CH^1(Y)\otimes \CH^{i}(Y)\ \to\ \CH^{i+1}(Y) \ ,\ \ \ 0\leq i<n\ ,\]
		is one-dimensional for every hypersurface $Y$ 
		of dimension $>2$ that is a quotient variety.
		Indeed in that case we have $\CH^1(Y) = \QQ [c_1(O_Y(1))]$ and $c_1(O_Y(1))
		\cdot \alpha$ is the restriction of a cycle on $\PP^{n+1}$ to~$Y$.
		In particular, for Calabi--Yau hypersurfaces of dimension $n \leq 4$ with
		quotient singularities, the image
		of
		the intersection product map
		\[ \CH^i(Y)\otimes \CH^{j}(Y)\ \to\ \CH^{i+j}(Y) \ ,\ \ \ i,j>0 \text{ and } i+j
		\leq n \]
		is one-dimensional.
	\end{remark}
	
	The following lemma is used in Remark~\ref{rmk:h}.
	
	\begin{lemma}\label{hyp} Let $\tau\colon Y\subset\PP^{n+1}(\C)$ be a
		hypersurface of degree $d$ that is a quotient variety. Let 
		$h\in \CH^1(Y)$ denote the restriction of $c_1(\OO_{\PP^{n+1}}(1))$ to $Y$.
		\begin{enumerate}[(i)]
\item The composition
\[ \CH^i(Y)\ \xrightarrow{\tau_\ast}\  \CH^{i+1}(\PP^{n+1})\
\xrightarrow{\tau^\ast}\
\CH^{i+1}(Y) \]
is the same as intersecting with $d h$.

\item Assume $n\ge 3$. Then $\CH^1(Y)=\QQ[h]$.
		\end{enumerate}
	\end{lemma} 
	
	\begin{proof}
		We recall that by convention, $\CH^i(Y)$ is identified with $\CH_{n-i}(Y)$
		(Lemma~\ref{quotient}). Point $(\rom1)$ is \cite[Proposition 2.6]{F}.
		As for $(\rom2)$, any (possibly singular) hypersurface $Y$ of dimension $\ge
		3$ has $\pic(Y)=\Z[h]$ (Grothendieck--Lefschetz). Because quotient varieties are
		$\QQ$-factorial, this implies that $\CH^1(Y)=\CH_{n-1}(Y)=\QQ[h]$.
	\end{proof}

	\begin{remark} It is possible to relax the hypotheses on the hypersurface $Y$
		in Theorem~\ref{cy}. Instead of demanding that $Y$ is a quotient variety, the
		argument proving Theorem~\ref{cy} goes through as soon as $Y$ is an {\em
			Alexander scheme\/} (in the sense of \cite{Vis}, \cite{Kim0}). In this case,
		$\CH^\ast(Y)$ (the operational Chow cohomology of Fulton \cite{F}) is isomorphic
		to the Chow groups $\CH_{n-\ast}(Y)$, and so the conclusion of Lemma
		\ref{quotient} holds for $Y$. Varieties with quotient singularities are examples
		of Alexander schemes.	
	\end{remark}

	\section{The Beauville--Voisin conjecture for zero-cycles on double EPW
		sextics}

	\subsection{Double EPW sextics and some of their geometry}\label{S:defEPW}
	As the name suggests, double EPW sextics are double covers of certain so-called
	EPW sextics\,:
	
	\begin{definition}[Eisenbud--Popescu--Walter \cite{EPW}] 
		Let $A\subset \wedge^3 \C^6$ be a subspace which
		is Lagrangian with respect to the symplectic form on $\wedge^3 \C^6$ given by
		the wedge product. The {\em EPW sextic associated to $A$\/} is
		\[ Y_A:= \Bigl\{  [v]\in \PP(\C^6)\ \vert\ \dim \bigl( A\cap ( v\wedge
		\wedge^2 \C^6)\bigr) \ge 1\Bigr\}\ \ \subset \PP(\C^6)\  .\]
		An {\em EPW sextic\/} is an $Y_A$ for some $A\subset \wedge^3 \C^6$
		Lagrangian.
	\end{definition}

	\begin{theorem}[O'Grady]
		Let $Y$ be an EPW sextic such that the singular locus $S:=\hbox{Sing}(Y)$ is a
		smooth irreducible surface. Let $f:X\to Y$ be the double cover branched over
		$S$. Then $X$
		is a smooth hyperK\"ahler fourfold of K3$^{[2]}$ type (called a {\em double
			EPW sextic}), 
		and the class $h:=f^*c_1(O_Y(1))\in \CH^1(X)$ defines a polarization of square
		$2$ for the Beauville--Bogomolov form.
		Double EPW sextics form a $20$-dimensional locally complete family.
	\end{theorem}
	
	\begin{proof} 
		This is \cite[Theorem 1.1(2)]{OG}. Let us remark that the hypothesis on
		$\hbox{Sing}(Y)$ is satisfied by a generic EPW sextic (more precisely, it
		suffices that the Lagrangian subspace $A$ be in $\hbox{LG}(\wedge^3 V)^0$, which
		is a certain open dense subset of $\hbox{LG}(\wedge^3 V)$ defined in
		\cite[Section 2]{OG}). Letting $A$ vary in $\hbox{LG}(\wedge^3 V)^0$, one
		obtains a locally complete family with $20$ moduli (as noted in
		\cite[Introduction]{OG}).		
	\end{proof}

	Let $Z$ be the invariant locus of the involution $\iota : X\to X$, which is also
	the pre-image of the singular locus $S$ of $Y$ along $f:X\to Y$. The following
	proposition summarizes the information that will be needed concerning $Z$ and
	its class in $\CH^2(X)$.
	
	\begin{proposition}\label{prop:Z}
		The  surface $Z$ is smooth projective, regular and Lagrangian. Moreover its
		class modulo rational equivalence satisfies the following relation
		\begin{equation}\label{eq:ferretti}
		3Z = 15h^2 - c_2(X)\quad  \text{in } \CH^2(X)\,.
		\end{equation}
	\end{proposition}
	\begin{proof}
		The surface $Z$ is 
		isomorphic to the singular locus of the EPW sextic $Y$, which is smooth
		irreducible by assumption.
		The fixed locus of an anti-symplectic involution on a hyperK\"ahler variety is
		(smooth and) Lagrangian \cite[Lemma 1]{Banti}.  The relation \eqref{eq:ferretti}
		is due to Ferretti \cite[Lemma~4.1]{Fe}.
		
		It remains to show that $Z$ is regular, \emph{i.e.} that $q(Z) := h^1(O_Z) = 0$. This is done in
		\cite[Corollary 3.19]{Fe0}.
(Alternatively, the irregularity and the geometric genus of $Z$ can
		also be computed using the Chow-theoretic results of \cite{V8}, as explained in
		\cite[Corollary 4.2$(\rom4)$]{La}.) 
	\end{proof}

	\subsection{Proof of Theorem~\ref{thm:ferretti}}
	
	We start by recalling the main result of \cite{Fe}, which establishes the
	Beauville--Voisin conjecture for the generic double EPW sextic.
	
	\begin{theorem}[Ferretti \cite{Fe}] \label{thm:ferretti-original}
		Let $X$ be a double EPW sextic. Consider the $\QQ$-subalgebra
		\[ \operatorname{R}^\ast(X):= \langle h, c_j(X)\rangle\ \ \ \subset\
		\CH^\ast(X)\
		\]
		generated by the polarization $h=f^*c_1(O_Y(1))$ and the
		Chern classes.
		The restriction of the cycle class map $\operatorname{R}^i(X)\to H^{2i}(X)$ is
		injective for all $i$. Moreover,
		\begin{equation}\label{eq:c2}
		c_2(X) \cdot h = 5h^3 \quad \text{in } \CH^3(X)\ .
		\end{equation}
	\end{theorem}
	
	\begin{proof} For illustrative purposes, let us briefly show how Ferretti's
		original argument can be slightly simplified by exploiting Theorem~\ref{cy}.
		
		First we observe that the tangent bundle $T_X$ is $\iota$-invariant, so that $\operatorname{R}^*(X)$ consists of $\iota$-invariant cycles. In degree 1, the injectivity is obvious.
		In degree 2, the injectivity follows from the fact that $h^2$ and $c_2(X)$ are
		linearly independent in $H^4(X)$. In degree 3, the generators of
		$\operatorname{R}^3(X)$ are $h^3$ and $h\cdot c_2(X)$ and, by Remark~\ref{rmk:h},  $h\cdot c_2(X)$ is	a multiple of $h^3$, thereby yielding the injectivity in degree 3. 
         In degree 4, the generators of
		$\operatorname{R}^4(X)$ are $h^4$, $h^2\cdot c_2(X)$, $c_2(X)\cdot c_2(X)$ and
		$c_4(X)$. Since $c_2(X)$ is $\iota$-invariant, Theorem~\ref{cy} yields that the subspace of  $\operatorname{R}^*(X)$ spanned by $h^4$, $h^2\cdot c_2(X)$ and $c_2(X)\cdot c_2(X)$ is one-dimensional. The injectivity in degree 4 now follows from the Chern class computation carried out by Ferretti in  \cite[Proposition 4.5]{Fe}, where the relation~\eqref{eq:c2} is also established.
	\end{proof}

	Consider now the eigenspace decomposition $$\CH^1(X) = \CH^1(X)^+ \oplus
	\CH^1(X)^-$$ for the action of the  involution $\iota$. Note that $\CH^1(X)^+ =
	f^*\CH^1(Y)$ is one-dimensional spanned by $h:= f^*c_1(O_Y(1))$ 
	and that $\CH^1(X)^-$ consists of primitive divisors. 
	The proof of Theorem~\ref{thm:ferretti} is a combination of Theorem~\ref{cy},
	which describes the intersection of $\iota$-invariant cycles on $X$ of positive
	and complementary codimension, and of the following lemma and elementary claim.
	
	\begin{lemma}\label{lem:key}
		Let $Z$ be the smooth Lagrangian surface which is the invariant locus of the
		involution $\iota : X\to X$.  Then $Z\cdot D = 0$ in $\CH^3(X)$ for all $D \in
		\CH^1(X)^-$.
	\end{lemma}
	\begin{proof} Denote $j: Z\hookrightarrow X$ the embedding.
		Since $Z$ is Lagrangian and defined for all double EPW sextics and since
		$H^2_{prim}(X) = H^2_{tr}(X)$ for the very general double EPW sextic, we have
		that $j^*H^2(X)_{prim} = 0$ and hence that $j^*\CH^1(X)^-$ consist of
		homologically trivial divisors on~$Z$. Since $Z$ is regular
		(Proposition~\ref{prop:Z}), we conclude that $j^*\CH^1(X)^- = 0$ and hence that
		$Z\cdot \CH^1(X)^- = 0$.
		
		(Alternative proof\,: since $Z\in \CH^2(X)^+$, we have $Z\cdot \CH^1(X)^-\subset
		\CH^3(X)^-$. On the other hand, $Z\cdot \CH^1(X)$ is generated by
		$j_\ast \CH^1(Z)$. But any divisor in $Z$ is $\iota$-invariant (as $Z$ is the
		fixed locus of $\iota$) and so $Z\cdot \CH^1(X)\subset \CH^3(X)^+$.
		The lemma follows from the fact that $\CH^3(X)^+\cap \CH^3(X)^-=0$.)
	\end{proof}

	\begin{claim}\label{claim}
		$h\cdot \alpha$ is a multiple of $h^3$ for all $\alpha \in \CH^2(X)^+$.
	\end{claim}
	\begin{proof}
		By the projection formula,
		the cycle $h\cdot \alpha$ is the pull-back along $f:X\to Y$ of the intersection
		of $c_1(O_Y(1))$ with the codimension-2 cycle $f_*\alpha$, and so
		(Remark~\ref{rmk:h}) is the pull-back along $f$ of a multiple of
		$c_1(O_Y(1))^3$, \emph{i.e.} it is a multiple of $h^3$. 
	\end{proof}
	
	\begin{proof}[Proof of Theorem~\ref{thm:ferretti}]
		First we observe that $c_2(X)$ is $\iota$-invariant (since $T_X$ is $\iota$-invariant) and that $c_4(X)$ belongs to the image
		of $\CH^2(X)^+ \otimes \CH^{2}(X)^+ \to \CH^4(X)$ (this follows from
		\cite[Proposition~4.5]{Fe}). 
		Next, let $D_k \in \CH^1(X)^-$ be anti-invariant divisors on $X$. 
		We note from Claim~\ref{claim} that, since $D_1\cdot D_2$ is $\iota$-invariant,
		$h\cdot D_1 \cdot D_2$ is a multiple of $h^3$.
		Thus $h\cdot D_1\cdot D_2 \cdot D_3$ is a multiple of $h^3\cdot D_3$.
		Intersecting Ferretti's relation \eqref{eq:ferretti} with $h\cdot D_3$,
		Lemma~\ref{lem:key} yields that $$ 0 = Z\cdot D_3 \cdot h = 15h^3\cdot D_3 -
		c_2(X)\cdot h \cdot D_3 = 15h^3\cdot D_3 - 5h^3\cdot D_3 = 10 h^3\cdot D_3,$$
		where the third equality comes from Ferretti's relation (\ref{eq:c2}). It follows that
		$\operatorname{R}^4(X)$ is spanned by cycles in $\mathrm{im} \big(\CH^i(X)^+
		\otimes \CH^{4-i}(X)^+ \to \CH^4(X) \big)$ for $i=1,2$. We conclude with
		Theorem~\ref{cy}.
	\end{proof}

	\begin{remark} The alternative proof of Lemma~\ref{lem:key} also shows that
		$Z\cdot \CH^2(X)^-=0$. Combined with Theorem~\ref{thm:ferretti},
		this implies that
		\[ Z\cdot \CH^2(X)=\QQ[c_4(X)]\ .\]
		This is an indication that $Z$ might perhaps be a constant cycle surface in
		$X$. 
	\end{remark}	
	
	\subsection{A variant of Theorem~\ref{thm:ferretti}} If one is willing to drop
	primitive divisors, then one can deal with codimension-3 cycles\,:
	\begin{theorem} \label{thm:ferretti2}
		Let $X$ be a double EPW sextic, and let $\iota$ be its anti-symplectic
		involution. Consider the $\QQ$-subalgebra
		\[ \operatorname{R}^\ast(X):= \langle h, \CH^2(X)^+, c_j(X)\rangle\ \ \
		\subset\ \CH^\ast(X)\
		\]
		generated by $h= f^*c_1(O_Y(1))$, $\iota$-invariant codimension-2 cycles and
		Chern classes.
		The restriction of the cycle class map $\operatorname{R}^i(X)\to H^{2i}(X)$ is
		injective for
		$i \geq 3$.
	\end{theorem}	
	\begin{proof}
		In view of Theorem~\ref{thm:ferretti}, we only need to prove the injectivity
		for $i=3$. Since $c_2(X)$ is $\iota$-invariant, this follows readily from Claim~\ref{claim}.
	\end{proof}

	\subsection{Towards the Beauville--Voisin conjecture for double EPW sextics}
	Since \cite{Bogomolov} the cup-product map $\mathrm{Sym}^2H^2(X) \to H^4(X)$ is injective for all
	hyperK\"ahler varieties of dimension larger than $2$, the Beauville--Voisin conjecture holds
	in codimension 2 for any hyperK\"ahler variety $X$ of dimension~$>2$ for which $[c_2(X)]$ does not
	belong to the image of the restriction $\mathrm{Sym}^2\operatorname{NS}(X) \to H^4(X)$  of the above cup-product map  to the N\'eron--Severi group of $X$. This is in particular 
	the case for hyperK\"ahler varieties of dimension $>2$ that are deformation equivalent to Hilbert schemes
	of K3 surfaces or to generalized Kummer varieties. 	The following proposition
	thus shows
	that the Beauville--Voisin conjecture for double EPW sextics reduces to showing
	that 
	\begin{equation}\label{eq:missing}
	(\CH^1(X)^-)^{\cdot 3} \subseteq h^2\cdot \CH^1(X)^-.
	\end{equation}

	\begin{proposition} Let $X$ be a double EPW sextic. Then
		$c_2(X)\cdot \CH^1(X) = h^2 \cdot \CH^1(X)$.
	\end{proposition}
	\begin{proof}
		Let $Z$ be the smooth Lagrangian surface which is the invariant locus of the
		involution $\iota : X\to X$.  By Lemma~\ref{lem:key}, we have $Z\cdot \CH^1(X)^-
		= 0$.
		Since $3Z = 15h^2 - c_2(X)$ in $\CH^2(X)$ (see~\eqref{eq:ferretti}), we have
		$c_2(X)\cdot \CH^1(X)^- = h^2\cdot \CH^1(X)^- $. Finally, the relation
		\eqref{eq:c2} concludes the proof of the proposition.
	\end{proof}
	
	Finally, let us mention that, due to Theorem~\ref{thm:ferretti} and precisely
	to the fact that $(\CH^1(X)^-)^{\cdot 4}$ injects in cohomology, it is likely
	that the recent result of Voisin \cite[Theorem 0.3]{VT} applies to double EPW
	sextics\,; this would imply that $(\CH^1(X)^-)^{\cdot 3}$ injects in cohomology.
	Although this would provide new information concerning the Beauville--Voisin
	conjecture in codimension 3 for double EPW sextics, it is however not clear how
	to establish the missing relation \eqref{eq:missing}. 
	
	\section{Hilbert squares of K3 surfaces with an involution}
	
	\subsection{MCK decomposition}
	\label{ssmck}
	Multiplicative Chow--K\"unneth decompositions were introduced in \cite[\S
	8]{SV} as a motivic way to provide an explicit candidate for Beauville's
	conjectural splitting of the conjectural Bloch--Beilinson filtration on the Chow
	rings of hyperK\"ahler varieties.
	First, we recall what a Chow--K\"unneth decomposition is.
	
	\begin{definition}[Murre \cite{Mur}] 
		Let $X$ be a smooth projective variety of
		dimension $n$. We say that $X$ has a {\em Chow--K\"unneth decomposition\/} (CK
		decomposition for short) if there exists a decomposition of the diagonal
		\[ \Delta_X= \Pi^0_X+ \Pi^1_X+\cdots +\Pi^{2n}_X\ \ \ \hbox{in}\ \CH^n(X\times
		X)\ ,\]
		such that the $\Pi^i_X$ are mutually orthogonal idempotents and
		$(\Pi^i_X)_\ast H^\ast(X)= H^i(X)$.
	\end{definition}
	
	Assuming the Bloch--Beilinson conjectures, Jannsen \cite{J2} proved that all
	smooth projective varieties admit a CK decomposition and moreover the CK
	projectors $\Pi^i_X$ induce a splitting of the Bloch--Beilinson filtration on
	the Chow \emph{groups}. A sufficient condition for the induced splitting to be
	compatible with the ring structure is given by the following definition.
	
	\begin{definition}[Shen--Vial \cite{SV}] 
		Let $X$ be a smooth projective variety
		of dimension $n$. Let $\delta_X \in \CH^{2n}(X\times X\times X)$ be the class
		of
		the small diagonal
		\[ \delta_X:=\bigl\{ (x,x,x)\ \vert\ x\in X\bigr\}\ \subset\ X\times X\times
		X\ .\]
		A \emph{multiplicative Chow--K\"unneth decomposition} (MCK decomposition for
		short) is a CK decomposition $\{\Pi^i_X\}$ of $X$ that is {\em
			multiplicative\/}, \emph{i.e.} that satisfies
		\[ \Pi^k_X\circ \delta_X \circ (\Pi^i_X\times \Pi^j_X)=0\ \ \ \hbox{in}\
		\CH^{2n}(X\times X\times X)\ \ \ \hbox{for\ all\ }i+j\not=k\ .\]
		An MCK decomposition is necessarily {\em self-dual\/}, \emph{i.e.} it
		satisfies
		$\Pi^k_X={}^t \Pi^{2n-k}_X$ for all $k$, where the superscript $t$ indicates
		the transpose correspondence\,; see \cite[footnote 24]{FV}.
	\end{definition}
	
	From the definition,
	it follows that if $X$ has an MCK decomposition $\{\Pi^i_X\}$, then setting
	\[ \CH^i_{(j)}(X):= (\Pi_X^{2i-j})_\ast \CH^i(X) \ ,\]
	one obtains a bigraded ring structure on the Chow ring\,: that is, the
	intersection product sends 
	$\CH^i_{(j)}(X)\otimes \CH^{i^\prime}_{(j^\prime)}(X) $ to 
	$\CH^{i+i^\prime}_{(j+j^\prime)}(X)$. In other words, an MCK decomposition induces a splitting
	of the conjectural Bloch--Beilinson filtration on the Chow ring of $X$.\medskip
	
	While the existence of a CK decomposition for any smooth projective variety is
	expected (either as part of the Bloch--Beilinson conjectures, or as part of
	Murre's conjectures \cite{Mur, J2}), the property of having
	an MCK decomposition is severely restrictive\,; for example, a very general
	curve of genus $\geq 3$ does not admit an MCK decomposition (although the
	conjectural BB filtration on the Chow ring of curves splits). The existence of
	an
	MCK decomposition is closely related to Beauville's ``weak splitting property''
	\cite{Beau3}, and it is conjectured in \cite[Conjecture~4]{SV} that
	hyperK\"ahler varieties should admit an MCK. 
	The seminal work of Beauville--Voisin \cite{BV} establishes for K3 surfaces $S$
	the existence of a canonical zero-cycle $o\in \CH^2(S)$ of degree 1 that
	``decomposes'' the small diagonal in $S\times S\times S$.
	By \cite[Proposition~8.14]{SV}, this can be reinterpreted as saying that the
	Chow--K\"unneth decomposition defined by $\Pi_S^0 := o\times S$, $\Pi_S^4 :=
	S\times o$ and $\Pi_S^2 = \Delta_S - \Pi_S^0 - \Pi_S^4$ is multiplicative.
	Beyond the case of K3 surfaces, the MCK conjecture for hyperK\"ahler varieties
	has been established
	for Hilbert squares of K3 surfaces in \cite{SV}, more generally for Hilbert
	schemes of length-$n$ subschemes on K3 surfaces in \cite{V6}, and for
	generalized
	Kummer varieties in \cite{FTV}. Other examples of varieties admitting an MCK
	can
	be found in \cite{SV2} and \cite{LV}. For more ample discussion and examples of
	varieties with an MCK decomposition, we refer to \cite[Section 8]{SV} and also
	\cite{V6}, \cite{SV2}, \cite{FTV}, \cite{LV}.

	\subsection{The Chow rings of birational hyperK\"ahler varieties} Consider two
	birational hyperK\"ahler varieties. We recall a result of Rie\ss~ showing the
	existence of a correspondence inducing an isomorphism between their Chow rings.
	
	\begin{theorem}[Rie\ss~\cite{Rie}]\label{rie} Let $\phi\colon X\dashrightarrow
		X^\prime$ be a birational map between two hyperK\"ahler varieties of dimension
		$n$. 
		\begin{enumerate}[(i)]
			\item  There exists a correspondence $R_\phi\in \CH^n(X\times X^\prime)$ such
			that 
			$(R_\phi)_\ast \colon \CH^\ast(X) \to \CH^\ast(X^\prime)$
			is a graded ring isomorphism.
			\item For any $j$, there is equality
			$(R_\phi)_\ast c_j(X)= c_j(X^\prime)$ in $\CH^j(X^\prime)$.
			\item There is equality
		    \[ (\Gamma_R)_\ast=(\bar{\Gamma}_\phi)_\ast\colon\ \ \CH^j(X)\ \to\
		\CH^j(X^\prime)\ \ \ \hbox{for}\ j=0,1, n-1, n\ \]
		    (where $\bar{\Gamma}_\phi$ denotes the closure of the graph of $\phi$). 		
		\end{enumerate}
	\end{theorem}
	
	\begin{proof} Item $(\rom1)$ is \cite[Theorem 3.2]{Rie}, while item $(\rom2)$ is
		\cite[Lemma 4.4]{Rie}.
		
		In a nutshell, the construction of the correspondence $R_\phi$ is as follows\,: \cite[Section 2]{Rie} provides a diagram 
		$$\xymatrix{\XX \ar[dr] \ar@{-->}[rr]^{\Psi} & & \XX' \ar[dl] \\
			& C
		}$$
	where $\XX, \XX^\prime$ are algebraic spaces over a quasi-projective
	curve $C$ such that the fibers $\XX_0, \XX^\prime_0$ are isomorphic to $X$
	resp. $X^\prime$, and where $\Psi\colon \XX
	\dashrightarrow \XX^\prime$ is a birational map inducing an isomorphism $\XX_{C\setminus 0}\cong  
	\XX^\prime_{C\setminus 0}$
	 and whose  restriction $\Psi\vert_{\XX_0}$
	coincides with $\phi$.	
	The correspondence $R_\phi$ is then defined as the specialization (in the sense
	of \cite{F}, extended to algebraic spaces) of the graphs of the isomorphisms
	$\XX_c\cong\XX^\prime_c$, $c\not=0$.      

Point $(\rom3)$ is not stated explicitly in \cite{Rie}, and can be seen as follows.
Let $U\subset X$ be the locus on which $\phi$ induces an isomorphism. The
		complement $T:=X\setminus U$ has codimension $\ge 2$. Using complete
		intersections of hypersurfaces, one can find a closed subset $\TT\subset\XX$ of
		codimension $2$ such that $\TT_0$ contains $T$, \emph{i.e.} $\Psi$ restricts to an
		isomorphism on $V:=\XX_0\setminus \TT_0$.  
		Since specialization commutes with pullback, the restriction of $R_\phi$ to
		$V\times X^\prime$ is the specialization of the graphs of the morphisms
		$\XX_c\setminus\TT_c\to \XX^\prime_c$, $c\not=0$ to $V\times X^\prime$, which
		is exactly the graph of the morphism $\phi\vert_V$, {\it i.e.}
		\[ R_\phi\vert_{V\times X\prime}= \Gamma_{(\phi\vert_V)}= \bar{\Gamma}_\phi
		\vert_{V\times X^\prime}\ \ \hbox{in}\ \CH^n(V\times X^\prime)\ .\]
		It follows that one has
		\[ R_\phi = \bar{\Gamma}_\phi + \gamma\ \ \ \hbox{in}\  \CH^n(X\times
		X^\prime)\ ,\]
		where $\gamma$ is some cycle supported on $\TT_0\times X^\prime$. This proves
		the ``moreover'' statement\,: $0$-cycles and $1$-cycles on $X$ can be moved to be
		disjoint of $\TT_0$, and so $\gamma_\ast$ acts as zero on $\CH^j(X)$ for $j\ge
		n-1$. Likewise, $\gamma^\ast$ acts as zero on $\CH^j(X^\prime)$ for $j\le 1$ for
		dimension reasons, and so (by inverting the roles of $X$ and $X^\prime$)
		statement $(\rom3)$ is proven. 
	\end{proof}

	\begin{remark}\label{birat} As observed in \cite[Introduction]{V6}, the
		correspondence $R_\phi$ of Theorem~\ref{rie} actually induces an isomorphism of
		Chow motives as $\QQ$-algebra objects. This implies that the property ``having
		an MCK decomposition'' is birationally invariant among hyperK\"ahler varieties.
	\end{remark}

	\subsection{MCK for $K3^{[2]}$}
	\label{secmck}
	
	Let $S$ be a K3 surface. We denote $S^{[2]}$ the Hilbert scheme of length-2
	subschemes of $S$ and  $Z:=\{(\zeta,x) \in S^{[2]}\times S : x\in
	\mathrm{Supp}(\zeta)\}$ the corresponding universal family. The former is
	naturally the quotient of the blow-up $\widetilde{S\times S}$ of $S\times S$
	along the diagonal under the natural involution switching the factors, while
	the
	latter comes equipped with two projection maps\,:
	$$ \xymatrix{ Z \ar[r]^{p \ \  } \ar[d]_q & S^{[2]} \ .\\
		S & }$$
	Recall that every divisor class on $S^{[2]}$ is of the form $p_*q^*D_S +
	a\delta$ for some divisor class $D_S$ on $S$ and some integer $a$, where
	$\delta$ denotes half of the image of the exceptional divisor under the
	quotient
	morphism $\widetilde{S\times S} \to S^{[2]}$.

	For a closed point $x \in S$,  $S_x := p(q^{-1}(x))$ defines a smooth
	subvariety of $S^{[2]}$ (which canonically identifies with the blow-up of $S$
	at
	$x$) and its class in $\CH^2(S^{[2]})$ is $Z_*[x]$. For two distinct points
	$x, y \in S$,  $[x, y]$  denotes the point of $S^{[2]}$ that corresponds to the
	subscheme $\{x, y\} \subset S$.
	When $x = y$, $[x, x]$ denotes the element in $\CH^4(S^{[2]})$ represented by
	any
	point corresponding to
	a nonreduced subscheme of length $2$ of $S$ supported at $x$. As cycles, we
	have $S_x \cdot S_y = [x, y]$.
	\medskip
	
	The following statement summarizes the results concerning the Chow ring of
	hyperK\"ahler varieties birational to the Hilbert square of a K3 surface that
	will be needed for the proofs of Theorems \ref{main0} and
	\ref{main}.
	
	\begin{theorem}[Shen--Vial \cite{SV}]\label{mck} 
		Let $S$ be a K3 surface, and
		let $X$ be a hyperK\"ahler fourfold birational to the Hilbert square
		$S^{[2]}$.
		Then $X$ admits a self-dual MCK decomposition such that the  induced bigraded
		ring structure $\CH^\ast_{(\ast)}(X)$ on $\CH^\ast(X)$ coincides with the
		bigrading
		on $\CH^\ast(X)$ defined by the ``Fourier transform'' of \cite{SV} and enjoys
		the
		following properties\,:
		\begin{enumerate}[(i)]
			\item \label{chern} $c_j(X)\in \CH^{j}_{(0)}(X)$ for all $j$\,;
			\item \label{hard} The multiplication map
			$\cdot D^2\colon \CH^2_{(2)}(X) \to \CH^4_{(2)}(X)$
			is an isomorphism for any choice of divisor $D\in \CH^1(X)$ with
			$\deg(D^4)\neq
			0$\,;
			\item \label{prod}
			The intersection product map
			$\CH^2_{(2)}(X)\otimes \CH^2_{(2)}(X) \to \CH^4_{(4)}(X)$
			is surjective.
		\end{enumerate}	
		Moreover, with respect to the choice of any  birational map
		$X\stackrel{\sim}{\dashrightarrow} S^{[2]}$,
		the bigraded pieces of the Chow groups have the following explicit
		descriptions\,:
		\begin{itemize}
			\item $\CH^0(X) = \CH^0_{(0)}(X) = \QQ [X]$\,;
			\item $\CH^1(X) = \CH^1_{(0)}(X)$ injects in cohomology \emph{via} the cycle
			class map\,;
			\item $\CH^2(X) = \CH^2_{(0)}(X) \oplus \CH^2_{(2)}(X)$, where $\CH^2_{(2)}(X) =
			\langle
			S_x - S_y : x,y\in S \rangle$\,;
			\item $\CH^3(X) = \CH^3_{(0)}(X) \oplus \CH^3_{(2)}(X)$, $\CH^3_{(2)}(X) =
			\CH^3_{hom}(X)$\,;
			\item $\CH^4(X) = \CH^4_{(0)}(X) \oplus \CH^4_{(2)}(X) \oplus \CH^4_{(4)}(X)$,
			where
			$\CH^4_{(0)}(X) = \QQ [o,o]$, $\CH^4_{(2)}(X) = \langle [x,o] - [y,o] :  x,y
			\in
			S\rangle$ and $\CH^4_{(4)}(X) = \langle [x,y] - [x,o] - [y,o] + [o,o]:  x,y
			\in
			S\rangle$.
		\end{itemize}
		Here, $o$ denotes any point lying on a rational curve on $S$.
	\end{theorem}
	
	\begin{proof}
		By the result of Rie\ss\  (Theorem~\ref{rie}), birational hyperK\"ahler
		varieties
		have isomorphic Chow motives as algebra objects (and Chern classes are sent to
		Chern classes). It follows that the proof of the theorem 
		reduces to the case of
		$X=S^{[2]}$\,; see \emph{e.g.} \cite[Section~6]{SV} and
		\cite[Introduction]{V6}.
		
		We consider the MCK on $X$ constructed in \cite[Theorem 13.4]{SV}\,; its
		relation with the Fourier transform is \cite[Theorem 15.8]{SV}. Statement~\eqref{chern}
		about the Chern classes is \cite[Lemma~13.7(iv)]{SV}, while statement~\eqref{prod} is
		\cite[Proposition~15.6]{SV}.
		The explicit description of the Chow groups is the combination of
		\cite[Theorem~2]{SV}, \cite[Proposition~15.6]{SV} and
		\cite[Proposition~12.9]{SV}.
		
		It remains to check \eqref{hard}. Let $D \in \CH^1(S^{[2]})$. Then $D =
		p_*q^*D_S + a\delta$, where $a\in \QQ$ and $D_S\in
		\CH^1(S)$. As in the proof of \cite[Proposition~12.8]{SV}, one computes
		$D^2\cdot
		S_x = -a^2[x,x] + \deg(D_S^2) [x,o]$. Combined with the fact that $[x,x] =
		2[x,o]- [o,o]$ (\cite[Proposition~12.6]{SV}) and the fact that $q(D) =
		\deg(D_S^2)  - 2a^2$, one finds 
		$$D^2\cdot (S_x -S_y) = q(D)([x,o]-[y,o]).$$
		Now if $D$ is such that  $\deg(D^4)\neq 0$, since $q(D)^2 = \lambda \deg(D^4)$
		with $\lambda$ the Fujiki--Beauville--Bogomolov constant (which is non-zero),
		we
		see that intersecting with $D^2$ induces an isomorphism  $\CH^2_{(2)}(X)\
		\stackrel{\sim}{\longrightarrow}\ \CH^4_{(2)}(X) $.
	\end{proof}

	\subsection{Proof of Theorem~\ref{main0}}
	
	Before proving Theorem~\ref{main0}, we first prove the following statement,
	which
	determines the action of $\iota$ on the relevant pieces of the Fourier
	decomposition $\CH^\ast_{(\ast)}(X)$.

	\begin{theorem}\label{main} Let $X$ be a smooth double EPW sextic, and assume
		that $X$ is birational to a Hilbert square $S^{[2]}$ with $S$ a K3 surface.
		Let
		$\iota\in\aut(X)$ be the anti-symplectic involution coming from the double
		cover
		$f \colon X\to Y$, where $Y\subset\PP^5$ is an EPW sextic.
		Then $\iota^*$ commutes with the Fourier decomposition, \emph{i.e.} $\iota^*$ 
		respects the grading of $\CH^i_{(*)}(X)$ given by Theorem~\ref{mck} for all
		$i$.
		Moreover, we have $\CH^1(X)^+ = \QQ [h]$ and $\CH^3_{(0)}(X)^+ = \QQ [h^3]$,
		together with
		\[    \begin{split}         \iota^\ast=\ide\colon&\ \ \ \CH^4_{(0)}(X)\ \to\
		\CH^4_{(0)}(X)\ ;\\      
		\iota^\ast =-\ide\colon&\ \ \ \CH^4_{(2)}(X)\
		\to\ \CH^4_{(2)}(X)\ ;\\
		\iota^\ast =\ide\colon&\ \ \ \CH^4_{(4)}(X)\
		\to\ \CH^4_{(4)}(X)\ ;\\
		\iota^\ast=-\ide \colon& \ \ \ \CH^2_{(2)}(X)\ \to\
		\CH^2_{(2)}(X)\ .\\           
		\end{split}\]
		
	\end{theorem}

	\begin{proof}
		Recall that $h=f^*c_1(O_Y(1))$\,; in particular, $h$ is $\iota$-invariant.
		Since $X$ satisfies the assumptions of Theorem~\ref{mck}, we know that 
		$\CH^4(X)$ is generated by $\CH^2_{(2)}(X)$ and by $h$. Therefore, Theorem
		\ref{main} for $\CH^4(X)$ would
		follow from  knowing
		that
		$\CH^2_{(2)}(X)$ is $\iota$-anti-invariant.
		Considering that $\CH^2_{(2)}(X)$ consists of homologically trivial cycles, the
		latter would follow from knowing that $\CH^2_{hom}(Y)=0$, where $Y$ is the EPW
		sextic. However, although it is conjectured that $\CH^2_{hom}(Y)=0$ for any
		hypersurface Y of dimension $\geq 4$, this is an intractable problem for
		Calabi--Yau hypersurfaces and for hypersurfaces of general type -- not a
		single
		case is known.
		Instead, our strategy consists in first showing that $\CH^4_{(2)}(X)$ is
		$\iota$-anti-invariant and then deduce that $\CH^2_{(2)}(X)$ is
		$\iota$-anti-invariant. 
		
		We now proceed to the proof of Theorem~\ref{main} and argue step-by-step. 
		Since $\CH^i(X) = \CH^i_{(0)}(X)$ for $i=0,1$, it is clear that $\iota$
		preserves the multiplicative grading in these cases. If $X$ is very general in
		moduli, then $\CH^1(X) = \QQ [h]$, and consequently $H^2(X) = h^\perp \oplus \QQ [h]$
		where $h^\perp$ is an irreducible polarized Hodge structure. Since  $\iota$ is
		anti-symplectic, we get that $\iota$ acts as $-\ide$ on $h^\perp$. This must
		hold for all smooth double EPW sextics. It follows that $\CH^1(X)^+ = \QQ [h]$.
		Alternatively, $\CH^1(X)^+ = f^*\CH^1(Y)$, but $\CH^1(Y) = \QQ [c_1(O_Y(1))]$ by the
		Grothendieck--Lefschetz theorem.
		
		By Theorem~\ref{mck}, $\CH^3(X) = \CH^3_{(0)}(X) \oplus \CH^3_{(2)}(X)$,
		where $\CH^3_{(2)}(X) = \CH^3_{hom}(X)$ is clearly stable under the action of
		$\iota$. In particular, $\CH^3_{(0)}(X)$ identifies with the image of the cycle
		class map $\CH^3(X) \to H^6(X)$. Therefore, by the hard Lefschetz theorem
		together with the multiplicativity of the bigrading on $\CH^*_{(*)}(X)$, we have
		$\CH^3_{(0)} (X)= h^2\cdot \CH^1(X)$. It follows readily that  $\CH^3_{(0)}(X)^+ =
		\QQ [h^3]$.

		We now deal with the remaining cases of codimension-4 and codimension-2
		cycles.
		First, note that $\CH^4_{(0)}(X)=\QQ [h^4]$ (apply Theorem~\ref{mck}), so that
		$\CH^4_{(0)}(X)$ is stable under the action of $\iota$, \emph{i.e.} we have
		$\CH^4_{(0)}(X)= \CH^4_{(0)}(X)^+$.
		Let now $\alpha\in \CH^2_{(2)}(X)$, and consider the 
		$\iota$-invariant zero-cycle $\alpha \cdot h^2+\iota^\ast(\alpha \cdot h^2)$.
		By the projection formula, we have 
		\[ \alpha \cdot h^2+\iota^\ast(\alpha \cdot h^2) = f^*\bigl((f_*\alpha)\cdot
		c_1(O_Y(1))^2\bigr)\ \ \in\ \CH^4(X).\]
		It follows readily from Remark~\ref{rmk:h}  that
		this cycle is a multiple of~$h^4$.
		Since $\alpha$ is algebraically trivial, the zero-cycles $\alpha\cdot h^2$ and
		$\iota^\ast(\alpha\cdot h^2)$
		are of degree zero, and hence
		\begin{equation}\label{zero} \alpha\cdot h^2+\iota^\ast(\alpha\cdot h^2)=0\ \
		\
		\hbox{in}\ \CH^4_{}(X)\ .\end{equation}
		As any element in $\CH^4_{(2)}(X)$ is of the form $\alpha \cdot h^2$ with
		$\alpha \in \CH^2_{(2)}(X)$  (Theorem~\ref{mck}\eqref{hard}),
		equality~(\ref{zero}) proves that $\CH^4_{(2)}(X) = \CH^4_{(2)}(X)^-$.
		
		Thus the action of $\iota$ commutes with the covariant action of $\Pi^6_X$ on
		$\CH^4(X)$. By Bloch--Srinivas~\cite{BS}, this means that 
		\begin{equation*}
		\iota^*\circ \Pi^6_X = \Pi^6_X\circ \iota^* + \Gamma \quad \hbox{in }
		\CH^4(X\times X)
		\end{equation*}
		for some correspondence $\Gamma$ supported on $D\times X$ for some divisor $D$
		in $X$. Let $\tilde{D} \to D$ be a desingularization of $D$. Since
		$\CH^2_{(2)}(X)$ consists of algebraically trivial cycles, and since
		$H^3(X) = 0$, the contravariant action $\Gamma^* : \CH^2_{(2)}(X) \to \CH^2(X)$
		factors through $\ker (AJ : \CH^1(\tilde D) \to \mathrm{Pic}^0(\tilde D))$,
		which
		is zero. Using that $\Pi^2_X$ is the transpose of $\Pi^6_X$, it follows that
		the action of $\iota$ and that of $\Pi^2_X$ commute on $\CH^2_{(2)}(X)$\,; in
		particular, the action of $\iota$ on $\CH^2(X)$ preserves $\CH^2_{(2)}(X)$.
		Equation~\eqref{zero} together with Theorem~\ref{mck}\eqref{hard} yields that
		$\alpha=-\iota^*\alpha$, proving that $\CH^2_{(2)}(X) = \CH^2_{(2)}(X)^-$.
		
		It then follows from Theorem~\ref{mck}\eqref{prod} that the action of $\iota$
		on $\CH^4(X)$ preserves $\CH^4_{(4)}(X)$ and that $\CH^4_{(4)}(X) =
		\CH^4_{(4)}(X)^+$.
		
		It remains to check that the action of $\iota$ on $\CH^2(X)$ preserves
		$\CH^2_{(0)}(X)$. Let $\alpha$ be a cycle in $\CH^2_{(0)}(X)$ and let us write
		$\iota^*\alpha = \beta_0 + \beta_2$, where $\beta_0 \in \CH^2_{(0)}(X)$ and
		$\beta_2 \in \CH^2_{(2)}(X)$. Since $h^2\cdot \CH^2_{(2j)}(X) \subseteq
		\CH^4_{(2j)}(X)$, and since  $\CH^4_{(0)}(X) = \CH^4_{(0)}(X)^+$ and $\CH^4_{(2)}(X)
		= \CH^4_{(2)}(X)^-$, we have
		$$\beta_0\cdot h^2 + \beta_2\cdot h^2 = \alpha\cdot h^2 = \iota^*(\alpha \cdot
		h^2) =  \iota^*(\beta_0\cdot h^2) + \iota^*(\beta_2\cdot h^2) =  \beta_0\cdot
		h^2 - \beta_2\cdot h^2.$$
		Consequently, $ \beta_2\cdot h^2  = 0$, and hence, thanks to
		Theorem~\ref{mck}\eqref{hard}, also $\beta_2 = 0$. Thus $\iota^*\alpha$ belongs
		to $\CH^2_{(0)}(X)$, thereby concluding the proof of the theorem.
	\end{proof}
	
	We now prove Theorem~\ref{main0} stated in the introduction\,:
	
	\begin{proof}[Proof of Theorem~\ref{main0}] Let us write $\phi\colon X^\prime\to
		X$ for the birational map from a Hilbert square $X^\prime$, and
		$\iota^\prime:=\phi^{-1}\circ\iota\circ\phi\in\bir(X^\prime)$ for the birational automorphism induced by
		$\iota\in\aut(X)$.
		
		To prove Theorem~\ref{main0} for $X^\prime$, it will suffice, thanks to the explicit generators for
		$ \CH^4_{(2j)}(X^\prime)$
		given in Theorem~\ref{mck}, to prove that
		\begin{equation}\label{+-}  (\iota^\prime)_\ast =(-1)^{j}\, \ide\colon\ \ \CH^4_{(2j)}(X^\prime)\ \to\
		\CH^4(X^\prime)\ .\end{equation}

		Let $R_\phi$ be Rie\ss's correspondence from Theorem~\ref{rie}\,; it induces,
		thanks to Remark~\ref{birat}, an isomorphism of bigraded
		rings  $(R_\phi)_\ast\colon\CH^\ast_{(\ast)}(X^\prime)
		\stackrel{\simeq}{\longrightarrow}	\CH^\ast_{(\ast)}(X)$.
		We claim that there is a commutative diagram
		\[ \xymatrix{
			\CH^4(X^\prime) \ar[r]^{(R_\phi)_\ast} \ar[d]^{\iota^\prime_\ast} &
			\CH^4(X) \ar[d]^{\iota_\ast} \\
			\CH^4(X^\prime) \ar[r]^{(R_\phi)_\ast}& \CH^4(X).\\
		}\]
		Clearly this claim, combined with Theorem~\ref{main}, establishes equality (\ref{+-}).
		To prove the claim, we use Theorem~\ref{rie}$(\rom3)$.
		We note that any zero-cycle $b\in \CH^4(X^\prime)$ can be
		represented by a cycle $\beta$ supported on an open $U^\prime \subseteq X'$ such
		that the birational map $\phi$ restricts to an
		isomorphism $\phi\vert_{U^\prime}\colon
		U^\prime\stackrel{\simeq}{\longrightarrow} U$ and $U\subseteq X$ is
		$\iota$-stable. Because $(R_\phi)_\ast$ and 
		$(\bar{\Gamma}_{\phi})_\ast$ coincide on $0$-cycles (Theorem~\ref{rie}$(\rom3)$),
		the image $(R_\phi)_\ast b\in \CH^4(Z)$ is supported on
		$(\phi\vert_{U^\prime})(\mathrm{Supp}(\beta)) \subset U$, which gives
		the claimed commutativity
		  \[    \iota_\ast (R_\phi)_\ast(b) =  (R_\phi)_\ast  (\iota^\prime)_\ast (b)\ \  \hbox{in}\ \CH^4(X)\ .\]
		\end{proof}

	\subsection{A concise reformulation of Theorem~\ref{main0}}
	In order to restate
	Theorem~\ref{main0} in a concise way, we invoke the following result\,:

	\begin{theorem} [Debarre--Macr\`i \cite{DM}]\label{dm}
		Let $S$ be a polarized
		K3 surface of degree $d$ and Picard number $1$, and let $X:=S^{[2]}$. Then
		$\bir(X)$ is trivial except in the following cases\,:
		
		\begin{itemize}
			\item $d=2$, or $d>2$ and $d$ 
			verifies 
			\[(\ast) \quad  a^2 d=2n^2+2\ ,\ \ \ a,n\ \in\Z\ ,\]		 
			while the Pell equation
			\[ \PPP_{2d}(5)\colon\ \ \  n^2-2d a^2=5 \]
			has no solution. In this case, $\aut(X)=\bir(X)=\Z/2\Z$.
			
			\item $d=10$, or $d$ is not divisible by $10$ and $d$ verifies
			$(\ast)$ and the Pell equation
			$\PPP_{2d}(5)$
			is solvable. In this case, $\aut(X)=0$ and $\bir(X)=\Z/2\Z$.
		\end{itemize}
		
		Moreover, if $\bir(X)$ is non-trivial and $d\not\in\{2,4,10\}$, $X$ is
		birational to a double EPW sextic.
	\end{theorem}
	
	\begin{proof} 
		This is \cite[Proposition B.3]{DM}. This extends and builds on
		prior work of Boissi\`ere \emph{et alii} \cite[Theorem 1.1]{BC}, who had proven the
		result for $\aut(X)$.   
	\end{proof}
	
	\begin{remark} 
		Double EPW sextics isomorphic to the Hilbert square of a K3 surface can be
		explicitly described \cite[Corollary 7.6]{DM} and are dense (for the euclidean
		topology) in the moduli space of double EPW sextics \cite[Proposition
		7.9]{DM}.
		Double EPW sextics birational to 	
		the Hilbert square of a K3 surface are explicitly described by Pertusi
		\cite[Theorem 1.4]{Per} (this completes earlier work of Iliev--Madonna
		\cite{IMad}). 
	\end{remark}

	\begin{corollary}\label{cor0}
		Let $X$ be a Hilbert scheme $X=S^{[2]}$ where
		$S$ is a K3 surface with $\pic(S)=\Z$. Let $\iota\in\bir(X)$ be a non-trivial
		birational automorphism. Then ($\iota$ is a non-symplectic birational
		involution, and) $\iota$ acts on $\CH^4(X)$ as in Theorem~\ref{main0}.
	\end{corollary}
	
	\begin{proof}
		Theorem~\ref{dm} implies that
		$\iota$ is the unique non-symplectic birational involution. In case the degree
		of
		$S$ is $2$, $\iota$ must be the automorphism induced by the covering
		involution
		$S\to\PP^2$, and it is elementary to prove that $\iota$ acts on $\CH^4(X)$ as
		requested. In case the degree of $S$ is $4$, $\iota$ must be the
		famous Beauville involution \cite{Beau0}, for which Theorem~\ref{main} (and
		hence Theorem~\ref{main0}) is known to hold (\cite[Corollary 1.8]{FLV} or
		\cite{Lat}). In case the degree of $S$ is $10$, $\iota$ must be the O'Grady
		involution \cite{OG2}.
		For this case, it was proven in \cite{epw} that $\iota$ acts as $-\ide$ on
		$\CH^4_{(2)}(X)$. Reasoning as in the proof of Theorem~\ref{main}, this proves
		that $\iota$ acts on $\CH^4(X)$ as in Theorem~\ref{main0} (a simpler, more
		geometric argument for the case of the O'Grady involution is given by Lin
		\cite{Lin}).
		Finally, in case the degree of $S$ is $\ge 6$ and not equal to $10$, $X$ must
		be birational to a double EPW sextic $X^\prime$ and $\iota$ is induced by the
		covering
		involution $\iota^\prime$ of the double EPW sextic (Theorem~\ref{dm}).
		It follows that $\iota$ acts on $\CH^4_{(j)}(X)$ as in Theorem~\ref{main0}.
	\end{proof}

	\begin{remark}\label{R:eq}
		We observe that Corollary~\ref{cor0} in turn implies Theorem~\ref{main0}.
		Indeed, assume $X$ is a double EPW sextic
		birational to a Hilbert square $S^{[2]}$.
		Let $d$ be the degree of $S$. Since "being birational to a Hilbert square" can
		be translated, \emph{via} lattice-theory, into a numerical condition on double
		EPW sextics (\emph{cf.} \cite{Per}), all Hilbert squares of degree-$d$ K3 surfaces are
		birational to a double EPW sextic (and $d$ satisfies the numerical condition of
		Theorem~\ref{dm}). Let $\XX_d\to \FF_d$ denote the universal family of Hilbert
		squares of degree-$d$ K3 surfaces. Corollary~\ref{cor0} applies to the very
		general element of this family. But then a standard spread argument (\cite[Lemma
		3.2]{Vo}), plus the fact that the graph of the covering involution and the MCK
		decomposition are universally defined, implies that Theorem~\ref{main} (and
		hence Theorem~\ref{main0}) holds for all elements of $\XX_d\to\FF_d$.	
	\end{remark}

	\subsection{Some applications of Theorem~\ref{main0}} \label{S:applications}
	We provide two corollaries to Theorem~\ref{main0}, which conjecturally should
	hold 
	for all double EPW sextics.
	First, when $X$ is birational to the Hilbert square of a K3 surface, we can
	improve Theorem~\ref{thm:ferretti} to the case of codimension-3 cycles.

	\begin{corollary}\label{cor} 	Let $X$ be a double EPW sextic, and assume either
		that $\CH^1(X) = \QQ [h]$, or that $X$ is birational to the Hilbert square of a K3
		surface. 
		Consider the subring
		\[ \operatorname{R}^\ast(X):= \langle \CH^1(X), \CH^2(X)^+, c_j(X)\rangle\ \ \
		\subset\
		\CH^\ast(X)\
		.\]
		The cycle class map $\operatorname{R}^i(X)\to H^{2i}(X)$ is injective for
		$i\ge 3$.
	\end{corollary}
	
	\begin{proof}
		The case where $\CH^1(X) = \QQ [h]$ is the content of
		Theorem~\ref{thm:ferretti2}.
		In the case $X$ is assumed to be birational to the Hilbert square of a K3
		surface,	 it suffices to show, due to Theorem~\ref{mck}, that $\CH^2 (X)^+\subset
		\CH^2_{(0)}(X)$.
		Since $\CH^2(X) = \CH^2_{(0)}(X) \oplus \CH^2_{(2)}(X)$, this follows from the
		facts proven in Theorem~\ref{main} that $\iota^*$ acts as
		$-1$ on $\CH^2_{(2)}(X)$ and that
		$\CH^2_{(0)}(X)$ is stable under the action of $\iota^*$.
	\end{proof}
	
	We note that	Corollary~\ref{cor} should hold for all $i$ (the problem in
	proving this is
	that
	it is not known whether $\CH^2_{(0)}(X)$ injects into cohomology) and for all
	double EPW sextics. 
	Second, still when $X$ is birational to the Hilbert square of a K3 surface or when
	$X$ is generic, we can characterize the canonical zero-cycle on $X$ as being the
	class of any point lying on a uniruled divisor whose class is $\iota$-invariant.
	\begin{corollary}\label{cor2}
		Let $X$ be a double EPW sextic, and assume either that $\CH^1(X) = \QQ [h]$, or
		that $X$ is birational to the Hilbert square of a K3 surface. Let $S_1
		\CH^4(X)\subset \CH^4(X)$
		denote the subgroup generated by points on uniruled divisors in $X$.  
		Then
		\[ S_1 \CH^4(X)\cap \CH^4(X)^+=\QQ[c_4(X)]\ .\]
	\end{corollary}
	
	\begin{proof} Let $X$ be a double EPW sextic. From \cite{Fe}, we know that
		$c_4(X)$ is a multiple of $h^4$ in $\CH^4(X)$, where $h = f^*c_1(O_Y(1))$ is the
		$\iota$-invariant polarization on $X$.
		From \cite[Theorem~1.6 and Remark 1.2]{CP}
		we have 
		\begin{equation}\label{eq:CP}
		S_1\CH^4(X) = \CH^1(X)\cdot \CH^3(X).
		\end{equation}
		First assume that $\CH^1(X) = \QQ [h]$. In that case, we have $S_1 \CH^4(X)\cap
		\CH^4(X)^+ = h\cdot \CH^3(X)^+$. The corollary in that case then follows
		immediately from Remark~\ref{rmk:h}.
		
		Second, assume that $X$ is birational to the Hilbert square of a K3 surface.
		In that case, the subgroup $S_1 \CH^4(X)$ coincides with $\CH^4_{(0)}(X)\oplus
		\CH^4_{(2)}(X)$, where $\CH^*_{(*)}(X)$ refers to the Fourier decomposition of
		Theorem~\ref{mck}. This can be seen either as a combination of \cite[End of
		Section~4.1]{V14} and the fact that $S_1\CH^4(X)$ is a birational invariant due
		to \cite{Rie} and the characterization~\eqref{eq:CP}, or as a direct consequence
		of Theorem~\ref{mck}.
		Theorem~\ref{main} then implies that the $\iota$-invariant part of $S_1
		\CH^4(X)$ is $\CH^4_{(0)}(X)\cong\QQ[h^4]$.
	\end{proof}

	\vskip1cm
	\begin{nonumberingt} 
		Thanks to IRMA Strasbourg for making possible the visit of the second author
		to Strasbourg  in February and August 2019. Thanks to Hsueh-Yung Lin, Laurent Manivel and Ulrike Rie\ss\ for helpful comments.
		\end{nonumberingt}

\end{document}